\documentclass[a4paper,10pt]{article}

\usepackage{amsmath,amssymb,amsthm,amscd}
\usepackage[mathscr]{eucal}
\usepackage{multirow}
\usepackage{hhline}
\usepackage[all]{xy}

\theoremstyle{plain}

\newcommand{\mcal}{\mathcal}

\newcommand{\mbb}{\mathbb}
\newcommand{\mrm}{\mathrm}

\newtheorem{theorem}{Theorem}[section]
\newtheorem{corollary}[theorem]{Corollary}
\newtheorem{lemma}[theorem]{Lemma}
\newtheorem{question}[theorem]{Question}
\newtheorem{proposition}[theorem]{Proposition}
\newtheorem{conjecture}[theorem]{Conjecture}

\theoremstyle{definition}
\newtheorem{definition}[theorem]{Definition}
\newtheorem{remark}[theorem]{Remark}

\newtheorem{example}[theorem]{Example}

\newtheorem*{acknowledgements}{Acknowledgements}

\title{Non-existence of certain Galois representations
with a uniform tame inertia weight}
\author{Yoshiyasu Ozeki\footnote{Supported by the JSPS Fellowships for 
Young Scientists.
\newline \quad \ \ e-mail: \texttt{y-ozeki@math.kyushu-u.ac.jp}}}

\date{}
\begin{document}
\maketitle

\begin{abstract}
In this paper, 
we prove the
non-existence of certain semistable Galois representations of 
a number field.
Our consequence can be applied to some geometric 
problems.
For example, we prove a special case of 
a Conjecture of Rasmussen and Tamagawa,
related with the finiteness of the set of isomorphism classes of 
abelian varieties with constrained prime power torsion.
\end{abstract}

\setcounter{section}{-1}
\section{Introduction}

Let $\ell$ be a prime number
and $K$ a number field.
In this paper, 
we show the non-existence of certain semistable 
$\ell$-adic Galois representations of the absolute Galois group
$G_K$ of $K$
by using remarkable results on the tame inertia weights
due to Caruso.
Fix non-negative integers $n,r$ and $w$, 
and a prime number $\ell_0\not =\ell$.
Put
$\bullet:=(n,\ell_0,r,w)$.
We consider the set $\mrm{Rep}_{\mbb{Q}_{\ell}}
(G_K)^{\bullet}$ of isomorphism classes of 
$\ell$-adic representations of $G_K$
(Definition \ref{Def1} (2)).
This set is related with 
the dual of
$H^w_{\mrm{\acute{e}t}}(X_{\bar K},\mbb{Q}_{\ell})$,
where  $X$ is a proper smooth scheme over $K$
which has everywhere semistable reduction
and has good reduction at a place of $K$ above $\ell_0$.
Our main result in this paper is 

\begin{theorem}
[= Theorem \ref{Cor2}]
\label{Intro}
Suppose that $w$ is odd or $w> 2r$.
Then there exists an explicit constant $C$ depending only on 
$K,n,\ell_0,r$ and $w$ such that 
$\mrm{Rep}_{\mbb{Q}_{\ell}}
(G_K)^{\bullet}$
is empty for any prime number $\ell>C$
which does not split in $K$.
\end{theorem}

\noindent
Theorem \ref{Intro} comes from a relation between 
the tame inertia weights and eigenvalues of Frobenius
action (Proposition \ref{Thm2}).
As a by-product of the above theorem,
we obtain some approaches to algebraic geometry. 
For example, 
our result gives an application to 
a special case of the 
Rasmussen-Tamagawa conjecture 
(\cite{RT})
related with the finiteness of the set of 
isomorphism classes of
abelian varieties with constrained prime power torsion.

Now we describe an organization of this paper.
In Section 1, we recall some results 
on integral $p$-adic Hodge theory
given by Caruso \cite{Ca}. 
In Section 2, we give explicit
values of the tame inertia weights for
certain semistable
Galois representations and prove 
our non-existence theorem. 
In Section 3, we apply our consequence
for some geometric problems.

\begin{acknowledgements}
The author wish to thank Shin Hattori
for bringing the work of Xavier Caruso.
The author would like to express his sincere gratitude to 
Akio Tamagawa and Seidai Yasuda who pointed out the mistake of 
the previous version of the proof for 
the main theorem and gave him useful advise.
\end{acknowledgements}

\vspace{2mm}

\noindent
{\bf Notation:}

\noindent
For a prime number $\ell$ and a topological group $G$, 
an {\it $\ell$-adic representation of $G$} 
(resp.\ {\it $\mbb{F}_{\ell}$}-representation) is a 
finite-dimensional $\mbb{Q}_{\ell}$-vector space $V$
(resp.\ $\mbb{F}_{\ell}$-vector space $V$) 
equipped with a continuous and linear $G$-action.
For such a representation $V$, 
we denote by $V^{\vee}$ the dual of $V$, that is, 
$V^{\vee}:=\mrm{Hom}_{\mbb{Q}_{\ell}}(V,\mbb{Q}_{\ell})$
(resp.\ $V^{\vee}:=\mrm{Hom}_{\mbb{F}_{\ell}}(V,\mbb{F}_{\ell})$)
with the $G$-action defined by 
$g.f(v):=f(g^{-1}.v)$ for $f\in V^{\vee}, g\in G$ and $v\in V$. 
For any scheme $X$ over a commutative ring $R$
and an $R$-algebra $R'$, 
we denote the fiber product $X\times_{\mrm{Spec}(R)} \mrm{Spec}(R')$ 
by $X_{R'}$.

\section{Tame inertia weights of semistable representations}

In this section, we recall the definition 
of the tame inertia weights (cf.\ \cite{Se}, Section 1)
and Caruso's work for the tame inertia weights of 
a residual representation of semistable 
Galois representations (cf.\ \cite{Ca}). 
Let  $K_{\lambda}$  be a complete discrete valuation field of characteristic zero 
with perfect residue field $k$ of positive characteristic $\ell$
and $G_{K_{\lambda}}$ its absolute Galois groups.
Let $e$ be the absolute ramification
index of $K_{\lambda}$.
The tame inertia weights of an $\ell$-adic semistable Galois representation 
of $G_{K_{\lambda}}$ with Hodge-Tate weights in $[0,r]$
have remarkable properties if $er<\ell-1$.
For example,
Serre conjectured in \cite{Se} that 
the tame inertia weights 
on the Jordan-H\"older quotients of a residual representation of the
$r$-th $\ell$-adic \'etale cohomology group 
$H^r_{\mrm{\acute{e}t}}(X_{\bar K_{\lambda}},\mbb{Q}_{\ell})$ of a 
proper smooth scheme $X$ over $K_{\lambda}$ 
are  between $0$ and $er$.
Caruso proved this Serre's conjecture in \cite{Ca}
by using the integral $p$-adic Hodge theory.   
As an another example, in \cite{CS},
Caruso and Savitt proposed the tame inertia polygon of  
an $\ell$-adic semistable Galois representation of $G_{K_{\lambda}}$, and
they showed that  
this polygon has good relations with 
the Hodge polygon and the Newton polygon introduced in \cite{Fo}.

\subsection{Tame inertia weights}

We denote by  
$I_{\lambda}$ the inertia subgroup of $G_{K_{\lambda}}$, 
$I_w$ its wild inertia subgroup 
and $I_t:=I_{\lambda}/I_w$ the tame inertia group.  
Let $V$ be an $h$-dimensional 
irreducible $\mbb{F}_{\ell}$-representation of $I_{\lambda}$
and fix a separable closure $\bar{\mbb{F}}_{\ell}$ of $\mbb{F}_{\ell}$. 
By the irreducibility, 
the action of $I_{\lambda}$ on $V$ factors through $I_t$ and thus 
we can regard $V$ as a representation of $I_t$. 
Applying Schur's lemma, we see that 
$\mbb{E}:=\mrm{End}_{I_t}(V)$ is the finite field of order ${\ell}^h$. 
Moreover, 
the representation $V$ inherits a structure of a $1$-dimensional 
$\mbb{E}$-representation of $I_t$ by the natural manner. 
This representation is given by a character 
$\rho\colon I_t\to \mbb{E}^{\times}$.
Choose any isomorphism 
$f\colon \mbb{E}\to \mbb{F}_{{\ell}^h}$ and consider the composition
$\rho_f\colon I_t\overset{\rho}{\to} 
\mbb{E}^{\times}\overset{f}{\to}\mbb{F}_{{\ell}^h}^{\times}$:

\vspace{-3mm}

\begin{center}
$\displaystyle \xymatrix{
I_t \ar_{\rho}[rr] \ar[drr] \ar^{\rho_f}@/^5mm/[rrr]
& & \mbb{E}^{\times} \ar@{_{(}->}[d] \ar_{f}^{\simeq}[r]& 
\mbb{F}_{{\ell}^h}^{\times}.\\
   & & GL(V)=\mrm{End}_{\mbb{F}_{\ell}}(V)^{\times} 
&  }$
\end{center}

\vspace{-3mm}

Denote by $\mu_{{{\ell}^h}-1}(\bar K_{\lambda})$ the set of 
$({{\ell}^h}-1)$-st roots 
of unity in a separable closure $\bar K_{\lambda}$ of $K_{\lambda}$.
Consider an isomorphism $\mu_{{{\ell}^h}-1}(\bar K_{\lambda})
\simeq \mbb{F}_{{\ell}^h}^{\times}$
coming from a surjection
$\mcal{O}_{\bar K_{\lambda}}\to \bar{\mbb{F}}_{\ell}$,
where $\mcal{O}_{\bar K_{\lambda}}$ is the integer ring of $\bar K_{\lambda}$,
and take the following fundamental character of level $h$:

\vspace{-5mm}

\begin{align*}
\theta_h\colon & I_t\to \mu_{q-1}(\bar K_{\lambda})\simeq 
\mbb{F}_{{\ell}^h}^{\times}.\\
& \sigma\mapsto \dfrac{\eta^{\sigma}}{\eta} 
\end{align*}

\vspace{-3mm}

\noindent
Here $\eta$ is a (${{\ell}^h}-1$)-st root of a uniformizer of $K_{\lambda}$.
It is easy to check that 
$\theta_h^{1+\ell+\cdots +{\ell}^{h-1}}=\theta_1$, 
$\theta_h^{{{\ell}^h}-1}=1$ 
and,
with respect to $h$ embeddings
$\mbb{F}_{{\ell}^h}\hookrightarrow \bar{\mbb{F}}_{\ell}$,
all the fundamental characters are given by 
$\theta_{h,0}(:=\theta_h),\theta_{h,1},
\theta_{h,2}, \dots ,\theta_{h,h-1}$, where
$\theta_{h,i}=\theta_{h,i-1}^{\ell}$ for $0\le i\le h-1$ and 
$\theta_{h,0}=\theta_{h,h-1}^{\ell}$.
It is known that $\theta_1^{e}$ coincides with the 
mod $\ell$ cyclotomic character
(\cite{Se}, Section 1.8, Proposition 8).
Since $I_t$ is pro-cyclic and 
$\mrm{Im}(\theta_h)=\mbb{F}_{{\ell}^h}^{\times}$, 
there exists an integer $n_f\in \{0,1,\dots , {{\ell}^h}-2\}$
such that $\rho_f=\theta_h^{n_f}$.
If we decompose 
$n_f=n_0+n_1\ell+n_2\ell^2+\dots +n_{h-1}\ell^{h-1}$ with 
integers $0\le n_i\le \ell-1$ 
for any $i$, then we can see that the set 
$\{n_0,n_1,n_2,\dots ,n_{h-1}\}$ is independent of the choice of 
$f$. 
\begin{definition} 
We call these numbers $n_0,n_1,n_2,\dots ,n_{h-1}$ 
the \textit{tame inertia weights of $V$}. 
In general, for any  $\mbb{F}_{\ell}$-representation $V$ of $I_{\lambda}$,
the tame inertia weights of $V$ are the numbers of the 
tame inertia weights of 
all the Jordan-H\"older quotients of $V$.
\end{definition} 

\begin{example}
Suppose that $k$ is algebraically closed.
Let $E$ be an elliptic curve over $K_{\lambda}$ with 
semistable reduction.
If $E$ has supersingular reduction,
assume $e=1$.
Then the tame inertia weights 
of $E[\ell]$ are $0$ and $e$
(cf.\ \cite{Se},\ Section 1,\ Proposition\ 11 and 12).
\end{example}

\begin{definition} 
\label{TIW}
Let $V$ be an $\ell$-adic representation of $G_{K_{\lambda}}$.
The tame inertia weights of $V$ is the tame inertia 
weights of a residual representation of $V|_{I_{\lambda}}$. 
\end{definition} 

The above definition is independent of the choice of a residual representation
of $V$ by the Brauer-Nesbitt theorem.

\begin{definition}
\label{UTIW}
Let $w$ be an integer with $0\le w<\ell-1$
and $V$ be an $n$-dimensional  
$\ell$-adic representation of $G_{K_{\lambda}}$.
Denote by $w_1\le w_2\le \cdots \le w_n$ 
all the tame inertia weights of $V$.
We say that $V$ is {\it of uniform tame inertia weight $w$} if 
$w_1=w_2=\cdots =w_n=w$.

\end{definition}

\subsection{Caruso's Result} 
Fix an integer $r\ge 0$
such that $er<\ell-1$.
We use the ring $S$ and the category 
$\mrm{Mod}_{/S_{\infty}}^{r,\Phi,N}$ 
of finite torsion $S$-modules equipped with some additional structures
as in Section 1 of \cite{Ca}
without giving the precise definitions. 
The category $\mrm{Mod}_{/S_{\infty}}^{r,\Phi,N}$
is just the category \underbar{$\mcal{M}$}$^r$ 
given in {\it op.cit.}
The category $\mrm{Mod}_{/S_{\infty}}^{r,\Phi,N}$ is an 
abelian category (cf.\ \cite{Ca}, Section 3.5).
We denote by $\mrm{Rep}_{\mbb{Z}_{\ell}}^{\mrm{st}}(G_{K_{\lambda}})^r$ 
(resp.\ $\mrm{Rep}_{\mbb{Z}_{\ell}}(G_{K_{\lambda}})_{\mrm{tors}}$) 
the category of $G_{K_{\lambda}}$-stable $\mbb{Z}_{\ell}$-lattices of  
semistable 
$\ell$-adic representations of $G_{K_{\lambda}}$ with 
Hodge-Tate weights in $[0,r]$ 
(resp.\ the category of finite torsion $\mbb{Z}_{\ell}$-modules 
with a continuous $G_{K_{\lambda}}$-action).
Denote by 
$\mrm{Mod}_{/S}^{r,\Phi,N}$ the category of 
strongly divisible modules over $S$
of weight $r$
(cf.\ \cite{Ca},\ Section\ 7.1).
There exist the following two contravariant functors 
\[
T_{\mathrm{st}}:
\mrm{Mod}_{/S}^{r,\Phi,N}
\to \mathrm{Rep}_{\mathbb{Z}_{\ell}}^{\mathrm{st}}
(G_{K_{\lambda}})^r
\]
\noindent
and 
\[
T_{\mrm{st}}\colon 
\mrm{Mod}_{/S_{\infty}}^{r,\Phi,N}\to 
\mrm{Rep}_{\mbb{Z}_{\ell}}(G_{K_{\lambda}})_{\mrm{tors}}
\]

\noindent
satisfying good properties. For example,

\vspace{2mm}

\noindent
(1) (cf.\ \cite{Ca}, Theorem 1.0.5) The 1st $T_{\mrm{st}}$ is an isomorphism,

\vspace{2mm}

\noindent
(2) (cf.\ \cite{Ca}, Theorem 1.0.4) The 2nd $T_{\mrm{st}}$ is exact 
and fully faithful, and 
its essential image is stable under taking sub-objects and quotient objects.

%

\vspace{2mm}

\noindent
If $\mcal{M}\in \mrm{Mod}_{/S_{\infty}}^{r,\Phi,N}$ is isomorphic to 
$S/{\ell}^{n_1}S\oplus S/{\ell}^{n_2}S\oplus\cdots \oplus S/{\ell}^{n_d}S$ 
as $S$-modules, then 
$T_{\mrm{st}}(\mcal{M})$ is isomorphic to 
$\mbb{Z}_{\ell}/{\ell}^{n_1}\mbb{Z}_{\ell}\oplus \mbb{Z}_{\ell}/{\ell}^{n_2}\mbb{Z}_{\ell}
\oplus \cdots \oplus\mbb{Z}_{\ell}/{\ell}^{n_d}\mbb{Z}_{\ell}$ as $\mbb{Z}_{\ell}$-modules
(\cite{Ca}, Proposition 6.4.5).
By the definition of strongly divisible modules, we see that, 
for any strongly divisible module $\tilde{\mcal{M}}$ and $n\ge 0$,
the quotient $\tilde{\mcal{M}}/{\ell}^n\tilde{\mcal{M}}$ is 
an object of
$\mrm{Mod}_{/S_{\infty}}^{r,\Phi,N}$
and  
the following diagram is commutative:
\begin{center}
$\displaystyle \xymatrix{
\mrm{Mod}_{/S}^{r,\Phi,N}
 \ar[rr]^{T_{\mrm{st}}} \ar[d]^{\mrm{mod}\ {\ell}^n}& & 
\mrm{Rep}_{\mbb{Z}_{\ell}}^{\mrm{st}}(G_{K_{\lambda}})^r \ar[d]^{\mrm{mod}\ {\ell}^n}\\
\mrm{Mod}_{/S_{\infty}}^{r,\Phi,N} \ar[rr]^{T_{\mrm{st}}} & & 
\mrm{Rep}_{\mbb{Z}_{\ell}}(G_{K_{\lambda}})_{\mrm{tors}}. }$
\end{center}
If $k$ is algebraically closed and 
$\mcal{M}\in \mrm{Mod}_{/S_{\infty}}^{r,\Phi,N}$ is a simple object,
then $T_{\mrm{st}}(\mcal{M})$ is an irreducible $\mbb{F}_{\ell}$-representation
of $G_{K_{\lambda}}$ and its tame inertia weights are between
$0$ and $er$ (\cite{Ca}, Theorem 1.0.3). 
By using the above facts, we can show the following
important theorem:

\begin{theorem}[\cite{Ca}]
\label{Ca}
Let $T_{\ell}\in \mrm{Rep}_{\mbb{Z}_{\ell}}^{\mrm{st}}(G_{K_{\lambda}})^r$ 
and $\bar{T}_{\ell}=T_{\ell}/{\ell}T_{\ell}$ its residual representation.
Then the tame inertia weights 
of $\bar{T}_{\ell}|_{I_{\lambda}}$ are 
between $0$ and $er$.
\end{theorem}

\begin{proof}
We may assume that 
$k$ is algebraically closed.
Choose the strongly divisible module
$\tilde{\mcal{M}}$ corresponding to $T_{\ell}$ via $T_{\mrm{st}}$.
Then $\mcal{M}:=\tilde {\mcal{M}}/{\ell}\tilde{\mcal{M}}$ is contained in
$\mrm{Mod}_{/S_{\infty}}^{r,\Phi,N}$ and  
$T_{\mrm{st}}(\mcal{M})$ is isomorphic to $\bar{T}_{\ell}$. 
We identify $T_{\mrm{st}}(\mcal{M})$ with $\bar{T}_{\ell}$.
Since the essential image of 
$
T_{\mrm{st}}\colon \mrm{Mod}_{/S_{\infty}}^{r,\Phi,N}\to 
\mrm{Rep}_{\mbb{Z}_{\ell}}(G_{K_{\lambda}})_{\mrm{tors}} 
$
is stable under sub-quotient, 
any Jordan-H\"{o}lder quotient of $\bar{T}_{\ell}$ is
isomorphic to the representation
of the form $T_{\mrm{st}}(\mcal{M}')$ for some 
$\mcal{M}' \in \mrm{Mod}_{/S_{\infty}}^{r,\Phi,N}$.
The object $\mcal{M}'$ is simple because 
the functor $T_{\mrm{st}}$ is exact and fully faithful. 
Therefore, we obtain the desired result.
\end{proof}

\begin{remark}
In fact, we do not need the assumption $er<\ell-1$
for Theorem \ref{Ca} (the case $er\ge \ell-1$ is trivial). 
\end{remark}

\section{Non-existence theorems}

In this section, 
we calculate the tame inertia weights of  
$\ell$-adic representations with certain 
geometric and filtration conditions  
for a prime number $\ell$ large enough.
As a result, we show the non-existence theorems 
of certain Galois representations. 

Let $K$ be a finite extension over $\mbb{Q}$ 
and 
fix an algebraic closure $\bar K$ of $K$.
We put $G_K:=\mrm{Gal}(\bar K/K)$, the absolute Galois group of $K$.
Let $\ell$ be a prime number.
For any finite place $v$ of $K$, 
we denote by $G_v$ and $I_v$ the decomposition group and 
its inertia subgroup at $v$, respectively. 
Furthermore, we denote 
by $e_v$ the absolute ramification index at $v$,
$q_v$ the order of the residue field of $v$
and $\mrm{Fr}_v$ the 
arithmetic Frobenius at $v$.
For a place $\lambda$ of $K$ above $\ell$, 
we identify $G_{\lambda}$ with the absolute Galois group $G_{K_{\lambda}}$of
a ${\lambda}$-adic completion $K_{\lambda}$ of $K$
via a fixed embedding $\bar K\hookrightarrow \bar K_{\lambda}$,
where $\bar K_{\lambda}$ is an algebraic closure of $K_{\lambda}$.
\begin{definition}
Let ${\lambda}$ be a place of $K$ above $\ell$
and $V$ an $\ell$-adic representation of $G_K$.
The {\it tame inertia weights of $V$ at ${\lambda}$} is the 
tame inertia weights of $V|_{G_{\lambda}}$
(cf.\ Definition \ref{TIW}).
For an integer $0\le w< \ell-1$,
we say that $V$ is {\it of uniform tame inertia weight $w$ at ${\lambda}$}
if $V|_{G_{\lambda}}$ is of uniform tame inertia weight $w$
(cf.\ Definition \ref{UTIW}).
\end{definition}

\subsection{Geometric and filtration conditions}

We define the set of representations we mainly 
consider throughout this section.
We fix non-negative integers $n,r,w$ and $\bar w$,
and a prime number $\ell_0$ different from $\ell$.
Let $\chi_{\ell}$ be 
the mod $\ell$ cyclotomic character. 
Take an $n$-dimensional $\ell$-adic representation 
$V$ of $G_K$
and denote by $\bar V$ its residual representation.
Now we consider the following 
geometric conditions (G-1), (G-2), (G-2)$'$ and (G-3), and  
filtration 
conditions (F-1) and (F-2):

\vspace{2mm}

\noindent 
(G-1) For any place ${\lambda}$ of $K$ above $\ell$,
the representation
$V|_{G_{\lambda}}$ is semistable 
and has Hodge-Tate weights in $[0,r]$.

\vspace{2mm}

\noindent
(G-2) For some places ${\lambda}_0$ of $K$ above $\ell_0$,
the representation
$V$ is unramified at ${\lambda}_0$ and  
the characteristic polynomial 
$\mrm{det}(T-\mrm{Fr}_{{\lambda}_0}|V)$ 
has rational integer coefficients.
Furthermore, there exists  non-negative integers 
$w_1(V), w_2(V), \dots , w_n(V)$ such that
$w_1(V)+w_2(V)+\cdots +w_n(V)\le \bar w$
and
the roots of the above characteristic polynomial 
have complex absolute values
$q_{{\lambda}_0}^{w_1(V)/2}, q_{{\lambda}_0}^{w_2(V)/2},\dots ,
q_{{\lambda}_0}^{w_n(V)/2}$
for every embedding $\bar{\mbb{Q}}_{\ell}$ into $\mbb{C}$.

\vspace{2mm}

\noindent
(G-2)$'$ The condition (G-2) holds
and  
$w_1(V)=w_2(V)= \cdots = w_n(V)=w$.

\vspace{2mm}

\noindent
(G-3) For any finite place ${\lambda}$ of $K$ not above $\ell$, 
the action of $I_{\lambda}$ on $\bar V$ is unipotent.

\vspace{2mm}

\noindent
(F-1) The representation $\bar V$ 
has a filtration of $G_K$-modules
\[
\{0\}= \bar V_0\subset \bar V_1 \subset \dots \subset
\bar V_{n-1}\subset \bar V_n=\bar V
\]
such that $\bar V_k$ has dimension $k$ for each $1\le k\le n$. 

\vspace{2mm}

\noindent
(F-2) The condition (F-1) holds.
Moreover,  
for each $1\le k\le n$, the $G_K$-action on the quotient 
$\bar V_k/\bar V_{k-1}$ is given by 
$g.\bar v=\chi_{\ell}^{a_{k}}(g)\bar v$ 
for some $0\le a_{k}\le \ell-2$.     

\vspace{2mm}

\noindent
If an $\ell$-adic representation $V$ satisfies the condition (F-1), 
then we say that $V$ is {\it of residually Borel}.
We note that
it is independent of the choice of 
a residual representation $\bar V$ of $V$ whether
the filtration conditions (F-1) and (F-2) 
hold or not. 
If $n=2$, then (F-1)
is equivalent to the condition that 
$\bar V$ is reducible.

\begin{example}
Suppose $w\le r$. 
Let $X$ be a proper smooth scheme over $K$
which has everywhere semistable reduction and has good reduction at some 
places of $K$ above $\ell_0$.
Then  the dual   
$H^w_{\mrm{\acute{e}t}}(X_{\bar K}, \mbb{Q}_{\ell})^{\vee}$ of 
the $w$-th $\ell$-adic \'etale cohomology group of 
$X$ satisfies 
the geometric conditions (G-1), (G-2)$'$ and (G-3). 
\end{example}

\begin{proposition}
\label{var}
Let $X$ be a proper smooth scheme over $K$
and $w$ an odd integer.
Denote by $S_X$ the finite set of prime numbers $p$
such that $X$ has bad reduction at some place of $K$ above $p$. 
Then, there exists a finite extension $L$ of $K$
such that, for any $\ell\notin S_X$,
the $\ell$-adic representation 
$H^w_{\mrm{\acute{e}t}}(X_{\bar L}, \mbb{Q}_{\ell})$ of $G_L$
is semistable at all finite places.
\end{proposition}

\noindent
In particular, we have the following:
Let $X$ and $L$ be as above. 
Fix a prime number $\ell_0\notin S_X$ and 
take a prime number $\ell$ such that 
$\ell\not =\ell_0$ and $\ell\notin S_X$. 
Then  $H^w_{\mrm{\acute{e}t}}(X_{\bar L}, \mbb{Q}_{\ell})^{\vee}$
satisfies $\mrm{(G}$-$\mrm{1)}$,  $\mrm{(G}$-$\mrm{2)}'$ and 
$\mrm{(G}$-$\mrm{3)}$ as a representation of $G_L$.

\begin{proof}[Proof of Proposition \ref{var}]
If we admit the semistable conjecture for $X$,
then we can prove this proposition easily.
However, we can obtain the desired result 
without using the semistable conjecture as below:
For any algebraic extension $K'$ of $K$,
denote by $S_{X,K'}$ the set of places of $K'$ 
which is above one of the prime  numbers in $S_X$.
Take any place $v\in S_{X,K}$.
By de Jong's alteration theorem (\cite{dJ}, Theorem 6.5), 
there exist a finite extension $K_v'$ of $K_v$, 
a proper strictly semistable scheme $\mcal{Y}^v$  
over $\mcal{O}_{K_v'}$ 
and a morphism $\mcal{Y}^v\to \mcal{X}$
compatible with 
$\mrm{Spec}(\mcal{O}_{K_v'})\to \mrm{Spec}(\mcal{O}_{K_v})$
such that the morphism 
$f\colon \mcal{Y}^v\to \mcal{X}_{\mcal{O}_{K_v'}}$ induced by 
the above morphism
is an \'etale alteration
(see also \cite{Ts}, Theorem A3).
Here $\mcal{X}$ is a proper flat model of $X_{K_v}$ over $\mcal{O}_{K_v}$.
Such a model always exists by the compactification theorem of Nagata.
Take any prime number $\ell'$.
If we denote by $f_{\ast}$ and $f^{\ast}$ the induced homomorphisms 
$H^w_{\mrm{\acute{e}t}}(\mcal{Y}^v_{\bar K_v'},\mbb{Q}_{\ell'})\to 
H^w_{\mrm{\acute{e}t}}(X_{\bar K_v},\mbb{Q}_{\ell'})$ 
and $H^w_{\mrm{\acute{e}t}}(X_{\bar K_v},\mbb{Q}_{\ell'})\to 
H^w_{\mrm{\acute{e}t}}(\mcal{Y}^v_{\bar K_v'},\mbb{Q}_{\ell'})$ 
respectively,
then the map $f_{\ast}\circ f^{\ast}$ is the 
multiplication by $\mrm{deg}(f)$.
In particular, 
the map $f^{\ast}$ is injective.
Thus we may consider that  
$H^w_{\mrm{\acute{e}t}}(X_{\bar K_v'},\mbb{Q}_{\ell'})$ is
a sub-representation of 
$H^w_{\mrm{\acute{e}t}}(\mcal{Y}^v_{\bar K_v'},\mbb{Q}_{\ell'})$.
Now take a finite extension $K(v)$ of $K$ 
and a place $w(v)$ of $K(v)$ above $v$ such that 
$K(v)_{w(v)}=K_{v}'$, 
where $K(v)_{w(v)}$ is the $w(v)$-adic completion of $K(v)$.
The existence of $K(v)$ and $w(v)$
is an easy consequence of \cite{La}, 
Chapter II, Section 2, Proposition 4.
We denote by $L$ the Galois closure, over $K$,
of the field generated by all $K(v)$.
Here $v$ runs through all the places of $K$ in $S_{X,K}$.
Now we take a prime number $\ell\notin S_X$.
It suffices to show that 
the $\ell$-adic representation 
$H^w_{\mrm{\acute{e}t}}(X_{\bar L}, \mbb{Q}_{\ell})$ of $G_L$
is everywhere semistable.
Take any finite place $w_L$ of $L$.
If $w_L\notin S_{X,L}$, then $X$ has good reduction at $w_L$
and in particular 
$H^w_{\mrm{\acute{e}t}}(X_{\bar L}, \mbb{Q}_{\ell})$ is semistable
at $w_L$.
Suppose $w_L\in S_{X,L}$.
We denote the restriction of $w_L$ to $K$ by $v$.
Take $\mcal{Y}^v$ and the place $w(v)$ of $K(v)$ as  above.
Furthermore, we take a place $w_L'$ of $L$ above $w(v)$.
Since the action of $I_{w_L'}$ is unipotent on 
$H^w_{\mrm{\acute{e}t}}(\mcal{Y}^v_{\bar L},\mbb{Q}_{\ell})$, we 
have that the action of $I_{w_L'}$ on 
$H^w_{\mrm{\acute{e}t}}(X_{\bar K},\mbb{Q}_{\ell})$ is unipotent, too.
Since the inertia subgroup $I_{w_L'}$ conjugates with $I_{w_L}$
by the element of $G_K$, 
we see that 
the action of $I_{w_L}$ on 
$H^w_{\mrm{\acute{e}t}}(X_{\bar K},\mbb{Q}_{\ell})$ is also unipotent,
that is, 
$H^w_{\mrm{\acute{e}t}}(X_{\bar L},\mbb{Q}_{\ell})$
is semistable at $w_L$.
This finishes the proof.
\end{proof}

\begin{definition}
\label{Def1}
Put 
$\circ:=(n,\ell_0,r,\bar w)$ and 
$\bullet:=(n,\ell_0,r,w)$.

\noindent
(1) We denote by $\mrm{Rep}_{\mbb{Q}_{\ell}}
(G_K)^{\circ}_{\mrm{cycl}}$
(resp.\
$\mrm{Rep}_{\mbb{Q}_{\ell}}
(G_K)^{\bullet}_{\mrm{cycl}}$) 
the set of isomorphism classes of $n$-dimensional 
$\ell$-adic representations $V$ of $G_K$ 
which satisfy (G-1), (G-2) and (F-2)
(resp.\ (G-1), (G-2)$'$ and (F-2)).

\noindent
(2)  We denote by $\mrm{Rep}_{\mbb{Q}_{\ell}}
(G_K)^{\circ}$ 
(resp.\ $\mrm{Rep}_{\mbb{Q}_{\ell}}
(G_K)^{\bullet}$ )
the set of isomorphism classes of $n$-dimensional 
$\ell$-adic representations $V$ of $G_K$ 
which satisfy (G-1), (G-2), (G-3) and (F-1)
(resp.\ (G-1), (G-2)$'$, (G-3) and (F-1)). 
\end{definition}

\noindent
Clearly, we have 

\[
\mrm{Rep}_{\mbb{Q}_{\ell}}
(G_K)^{\circ}_{\mrm{cycl}}
\quad  \subset  \quad 
\mrm{Rep}_{\mbb{Q}_{\ell}}
(G_K)^{\circ}\\
\]
\[
\cup \qquad \qquad  \qquad \qquad \cup 
\]
\[
\mrm{Rep}_{\mbb{Q}_{\ell}}
(G_K)^{\bullet}_{\mrm{cycl}}
\quad  \subset  \quad  
\mrm{Rep}_{\mbb{Q}_{\ell}}
(G_K)^{\bullet},\\
\]
\noindent
where $\bullet=(n,\ell_0, r,w)$
and $\circ=(n,\ell_0,r,\bar w)$ for any $nw\le \bar w$.

Our main concern in this section is the following question:
\begin{question}
Does there exist a constant $C$
which depends on $K$ and $\bullet$ $($or $\circ$ $)$
such that the sets defined in Definition \ref{Def1}
are empty for $\ell>C$? 
If the answer is positive,
how can we evaluate such a constant $C$? 
\end{question}

\begin{remark}[Trivial case]
Take a representation  $V\in \mrm{Rep}_{\mbb{Q}_{\ell}}
(G_K)^{\bullet}$.
By (G-2), 
the complex absolute value of 
the determinant of $\mrm{Fr}_{v_0}$ acting on $V$ is 
$q_{v_0}^{nw/2}$ and this must be an integer.
From this fact, if $n$ and $w$ are odd and 
the extension $K/\mbb{Q}$ is Galois of an odd degree,
then $\mrm{Rep}_{\mbb{Q}_{\ell}}(G_K)^{\bullet}$ is empty
for any prime $\ell\not =\ell_0$.
As this example, there exist lots of pairs of 
$(K,\bullet )$ (resp.\ $(K,\circ)$)
such that $\mrm{Rep}_{\mbb{Q}_{\ell}}(G_K)^{\bullet}$
(resp.\ $\mrm{Rep}_{\mbb{Q}_{\ell}}(G_K)^{\circ}$) is empty
for a prime $\ell$ (large enough).
We hope to know ``non-trivial cases'' of the emptiness of 
the sets given in Definition \ref{Def1}.
\end{remark}

\subsection{Main results}

We denote by 
$d,d_K$ and $h_K^{+}$ the extension degree of $K$ over $\mbb{Q}$,
the discriminant of $K$ and the narrow class number of $K$,
respectively.
Put 
$M:=\mrm{max}\{nr,\bar w/2 \}$
and
\begin{align*}
c_n:=
\left\{
\begin{array}{cl}
\left(\begin{smallmatrix}
  n \\ n/2
\end{smallmatrix}\right)
 &\quad 
\mrm{if}\ n\ \mrm{is\ even}, \cr
\left(\begin{smallmatrix}
  n \\ (n-1)/2
\end{smallmatrix}\right)
 &\quad \mrm{if}\ n\ \mrm{is\ odd}.
\end{array}
\right.
\end{align*}
Clearly this is equal to 
$\mrm{max}\{\left(\begin{smallmatrix}
  n \\ m
\end{smallmatrix}\right)
\mid 0\le m \le n
\}$.
Now we put
\begin{align*}
&\varepsilon_1:=dM,\quad \varepsilon_2:=d\varepsilon_1,\quad
\varepsilon_1':=dh_K^+M,\quad 
\varepsilon_2':=d\varepsilon_1',\\
& C_1:=C_1(d,\bullet):=
2c_n\ell^{\varepsilon_1}_0,\quad 
C_2:=C_2(d,\bullet):=
2c_n\ell^{\varepsilon_2}_0,\\
& C_1':=C_1'(K,\bullet):=2c_n \ell_0^{\varepsilon_1'}, \quad 
C_2':=C'_2(K,\bullet):=2c_n \ell_0^{\varepsilon_2'}.
\end{align*}

The following two propositions play an essential role
for our main results.
\begin{proposition}
\label{Thm1}
Any $\ell$-adic representation $V$ 
in the set 
$\mrm{Rep}_{\mbb{Q}_{\ell}}
(G_K)^{\circ}_{\mrm{cycl}}$
has  tame inertia weights $e_{\lambda}w_1(V)/2,e_{\lambda}w_2(V)/2,\dots , e_{\lambda}w_n(V)/2$ 
at any place ${\lambda}$ of $K$ above $\ell$
under any one of the following
situations:

$\mrm{(a)}$ 
$\ell\nmid d_K$ and $\ell > C_1$;

$\mrm{(b)}$ 
$\ell > C_2$.
\end{proposition}

\begin{proposition}
\label{Thm2}
Suppose that $\ell$ is a prime number which does not split 
in $K$.
Any $\ell$-adic representation $V$
in the set 
$\mrm{Rep}_{\mbb{Q}_{\ell}}
(G_K)^{\circ}$
has  tame inertia weights $e_{\lambda}w_1(V)/2,e_{\lambda}w_2(V)/2,\dots , e_{\lambda}w_n(V)/2$ 
at the unique place ${\lambda}$ of $K$ above $\ell$
under any one of the following
situations:

$\mrm{(a)}$ 
$\ell\nmid d_K$ and $\ell > C'_1$;

$\mrm{(b)}$ 
$\ell > C'_2$.

\end{proposition}

\noindent
To prove these propositions,
we need the following lemma:

\begin{lemma}
\label{Lem0}
Let $s,t_1,t_2,\dots ,t_n$ and $u$ be 
non-negative integers 
such that $0\le s\le u$ and $0\le t_k\le ru$ for all $k$.
Let $V$ be an $n$-dimensional $\ell$-adic representation 
of $G_K$ which satisfies $\mrm{(G}$-$\mrm{2)}$.
Decompose $\mrm{det}(T-\mrm{Fr}_{{\lambda}_0}|V)=
\prod_{1\le k\le n}(T-\alpha_k)$.
If the set $\{\alpha_1^{s},\alpha_2^{s},\dots ,\alpha_n^{s}\}$
coincides with the set 
$\{q_{{\lambda}_0}^{t_1},q_{{\lambda}_0}^{t_2},\dots,q_{{\lambda}_0}^{t_n}\}$
in $\bar{\mbb{F}}_{\ell}$
and $\ell>2c_n\ell_0^{dMu}$,
then $\{\alpha_1^{s},\alpha_2^{s},\dots ,\alpha_n^{s}\}
=\{q_{{\lambda}_0}^{t_1},q_{{\lambda}_0}^{t_2},\dots,q_{{\lambda}_0}^{t_n}\}$.
In particular, 
we obtain
$\{sw_1(V)/2,sw_2(V)/2,\dots ,sw_n(V)/2\}=
\{t_1,t_2,\dots ,t_n\}$.
\end{lemma}
\begin{proof}
We basically follow the proof 
by the method 
which has been pointed out by 
Rasmussen and Tamagawa.
Let us denote by $S_m(x_1, x_2,\dots , x_n)$ 
the elementary symmetric polynomial 
of degree $m$ with $n$-indeterminates $x_1, x_2,\dots ,x_n$ 
for $0\le m\le n$, that is, 
\[
\prod_{1\le k\le n}(T-x_k)=
\sum_{0\le m\le n}S_m(x_1, x_2,\dots , x_n)T^{n-m}.
\]
For any $0\le m\le n$, the condition (G-2) implies that 
$S_m(\alpha_1,\alpha_2,\dots,\alpha_n)$ is a rational integer for all $m$
and hence $S_m(\alpha_1^{s},\alpha_2^{s},\dots,\alpha_n^{s})$,
which is a symmetric polynomial of $\alpha_1,\alpha_2,\dots , \alpha_n$, 
is also a rational integer. 
On the other hand,
we have 
\begin{align*}
|S_m(\alpha_1^{s}, \alpha_2^{s},\dots ,\alpha_n^{s})|
&\le 
\sum_{1\le s_1< \cdots < s_m\le n}
(q_{{\lambda}_0}^{(w_{s_1}(V)+\cdots +w_{s_m}(V))/2})^s\\
&\le
\sum_{1\le s_1< \cdots < s_m\le n}
(q_{{\lambda}_0}^{\bar w/2})^s
=
\left(\begin{smallmatrix}
  n \\ m
\end{smallmatrix}\right)
(q_{{\lambda}_0}^{\bar w/2})^s
\le c_n\ell_0^{dMu}  
\end{align*}
and 
\begin{align*}
|S_m(q_{{\lambda}_0}^{t_1},q_{{\lambda}_0}^{t_2},\dots ,q_{{\lambda}_0}^{t_n})|
&\le 
\sum_{1\le s_1< \cdots < s_m\le n}
q_{{\lambda}_0}^{t_{s_1}+\cdots +t_{s_m}}\\
&\le
\sum_{1\le s_1< \cdots < s_m\le n}
q_{{\lambda}_0}^{nru}
=
\left(\begin{smallmatrix}
  n \\ m
\end{smallmatrix}\right)
q_{{\lambda}_0}^{nru}\le c_n\ell_0^{dMu}
\end{align*}
\noindent
by (G-2),
where $|\cdot|$ is the complex absolute value.
Since we have 
$S_m(\alpha_1^{s}, \alpha_2^{s},\dots ,\alpha_n^{s})
\equiv 
S_m(q_{{\lambda}_0}^{t_1},q_{{\lambda}_0}^{t_2},\dots ,q_{{\lambda}_0}^{t_n})$
mod $\ell$
and $\ell> 2c_n\ell_0^{dMu}$, 
we obtain
\[
S_m(\alpha_1^{s}, \alpha_2^{s},\dots ,\alpha_n^{s})
=S_m(q_{{\lambda}_0}^{t_1},q_{{\lambda}_0}^{t_2},\dots ,q_{{\lambda}_0}^{t_n})
\]
for all $m$.
This implies
\begin{align*}
\prod_{1\le k\le n}(T-\alpha_k^{s})&=
\sum_{0\le m\le n}
S_m(\alpha_1^{s},\alpha_2^{s},\dots ,\alpha_n^{s})T^{n-m}\\
&=\sum_{0\le m\le n}
S_m(q_{{\lambda}_0}^{t_1},q_{{\lambda}_0}^{t_2},\dots ,q_{{\lambda}_0}^{t_n})T^{n-m}\\
&=\prod_{1\le k\le n}(T-q_{{\lambda}_0}^{t_k})
\end{align*}
and thus we finish the proof.
\end{proof}

Now we start the proofs of 
Proposition \ref{Thm1} and \ref{Thm2}. 
Take any representation $V$ which is an element of the set 
$\mrm{Rep}_{\mbb{Q}_{\ell}}
(G_K)^{\circ}$
and denote its residual representation by $\bar V$.
Then the representation 
$\bar V$ has a filtration of $G_K$-modules 
\[
\{0\}= \bar V_0\subset \bar V_1 \subset \dots \subset
\bar V_{n-1}\subset \bar V_n=\bar V
\]
such that $\bar V_k$ has dimension $k$ for each $1\le k\le n$.
We denote by $\psi_k\colon G_K\to \mbb{F}_{\ell}^{\times}$ 
the character corresponding to the 
action of $G_K$ on the quotient
$\bar{V}_k/\bar{V}_{k-1}$ for each $1\le k\le n$.
Take any place ${\lambda}$ of $K$ above $\ell$.
By Theorem \ref{Ca},
we obtain 
$\psi_k=\theta_{1,{\lambda}}^{b_{k,\lambda}}$ on $I_{\lambda}$ 
for some integer $0\le b_{k,\lambda}\le e_{\lambda}r$,
where $\theta_{1,{\lambda}}\colon I_{\lambda}
\to \mbb{F}_{\ell}^{\times}$ is the 
fundamental character of level one at ${\lambda}$.
Take a place ${\lambda}_0$ of $K$ above $\ell_0$ as in (G-2) and 
decompose $\mrm{det}(T-\mrm{Fr}_{{\lambda}_0}|V)=
\prod_k(T-\alpha_k)$.
Then, we see 
\[
\{\alpha_1,\alpha_2,\dots ,\alpha_n\}
=\{\psi_1(\mrm{Fr}_{{\lambda}_0}),
\psi_2(\mrm{Fr}_{{\lambda}_0}),\dots ,\psi_n(\mrm{Fr}_{{\lambda}_0})\} 
\qquad (\ast) 
\] 
in $\bar{\mbb{F}}_{\ell}$. 

\begin{proof}[Proof of Proposition \ref{Thm1}]
Assume that $V$ is an element of the set 
$\mrm{Rep}_{\mbb{Q}_{\ell}}
(G_K)^{\circ}_{\mrm{cycl}}$.
Then we may suppose  
$\psi_k=\chi_{\ell}^{a_{k}}$ for any $k$
by (F-2).
The relation 
$\chi_{\ell}^{a_{k}}=
\theta_{1,\lambda}^{b_{k,\lambda}}$ on $I_{\lambda}$ 
implies 
$\theta_{1,\lambda}^{e_{\lambda}a_{k}}=
\theta_{1,\lambda}^{b_{k,\lambda}}$ and thus  
$e_{\lambda}a_{k}\equiv b_{k,\lambda}$ mod $\ell-1$.
Hence we have 
$\chi_{\ell}^{e_{\lambda}a_{k}}=
\chi_{\ell}^{b_{k,\lambda}}$ on $G_K$
and thus the set
$\{\alpha_1^{e_{\lambda}}, \alpha_2^{e_{\lambda}},
\dots ,\alpha_n^{e_{\lambda}}\}$
coincides with the set 
$\{q_{{\lambda}_0}^{b_{1,{\lambda}}},
q_{{\lambda}_0}^{b_{2,\lambda}},\dots ,q_{\lambda_0}^{b_{n,\lambda}}\}$
in $\bar{\mbb{F}}_{\ell}$ by $(\ast)$.
By Lemma \ref{Lem0},
we have 
\[
\{e_{\lambda}w_1(V)/2,\dots, e_{\lambda}w_n(V)/2\}
=\{b_{1,\lambda},\dots, b_{n,\lambda}\}
\]
if $\ell>2B_n\ell_0^{dMe_{\lambda}}$.
Since
$e_{\lambda}\le d$ and $e_{\lambda}=1$ if $\ell \nmid d_K$,
we have the desired result.
\end{proof}

\begin{proof}[Proof of Proposition \ref{Thm2}]
We note that each $\psi_k$ is unramified away from $\ell$ by (G-3). 
Now we assume that any one of the following conditions 
(A) or (B) holds: 

\vspace{1mm}

\noindent 
$\mrm{(A)}$ $\ell\nmid d_K$;

\vspace{1mm}

\noindent
$\mrm{(B)}$ No additional assumptions.

\vspace{1mm}

\noindent
Setting $b'_k:=b_{k,{\lambda}}/e_{\lambda}\in \mbb{Q}$, 
we have $0\le b'_k\le r$.  
We note that, 
if we put 
\begin{align*}
D:=
\left\{
\begin{array}{cl}
1
 &\quad 
 \mrm{under\ (A)}, \cr
d
 &\quad 
 \mrm{under\ (B)},
\end{array}
\right.
\end{align*} 
then we see $D/e_{\lambda}\in \mbb{Z}$. 
Since $\psi_k=\theta_{1,{\lambda}}^{b_{k,{\lambda}}}$ on $I_{\lambda}$,
we see that 
$\psi_k^{e_{\lambda}}\chi_{\ell}^{-b_{k,{\lambda}}}$ 
is trivial on $I_{\lambda}$ and thus
$(\psi_k^{e_{\lambda}}\chi_{\ell}^{-b_{k,{\lambda}}})^{D/e_{\lambda}}=
\psi_k^D\chi_{\ell}^{-b'_kD}$ is also trivial 
on $I_{\lambda}$.
Since the characters $\psi_k$ and $\chi_{\ell}$ are unramified 
away from $\ell$, this implies that 
$\psi_k^{D}\chi_{\ell}^{-b'_kD}$ is unramified 
at all finite places of $K$ 
(recall that $\ell$ does not split in $K$).
By class field theory, it follows 
\[
\psi_k^{Dh_K^+}=\chi_{\ell}^{b'_kDh_K^+}
\]
on $G_K$.
Recall that  $h^+_K$ is the narrow class number of $K$.
Thus we have that 
the set 
$\{\alpha_1^{Dh_K^+},\alpha_2^{Dh_K^+},\dots ,\alpha_n^{Dh_K^+}\}$
coincides with the set 
$\{q_{{\lambda}_0}^{b'_1Dh_K^+},q_{{\lambda}_0}^{b'_2Dh_K^+},
\dots ,q_{{\lambda}_0}^{b'_nDh_K^+}\}$ in 
$\bar{\mbb{F}}_{\ell}$ by $(\ast)$.
Now we assume $\ell>2c_n\ell_0^{dDh_K^+M}$.
Then we have 
\[
\{Dh_K^+w_1(V)/2,\dots, Dh_K^+w_n(V)/2\}
=\{b'_1Dh_K^+,\dots, b'_nDh_K^+\}
\]
by Lemma \ref{Lem0}.
Our result comes from this equation. 
\end{proof}

Now we can obtain our main results.

\begin{theorem}
\label{Cor1} Suppose that $w$ is odd or $w> 2r$.
Then  the set 
$\mrm{Rep}_{\mbb{Q}_{\ell}}
(G_K)^{\bullet}_{\mrm{cycl}}$
is empty under any one of the following situations:

$\mrm{(a)}$ 
$w$ is odd, $\ell\nmid d_K$ and $\ell > C_1$;

$\mrm{(b)}$ 
$w$ is odd, the extension $K/\mbb{Q}$ has  odd degree 
and  $\ell > C_2$;

$\mrm{(c)}$ 
$w>2r$, $\ell\nmid d_K$ and $\ell > C_1$;

$\mrm{(d)}$
$w>2r$ and $\ell> C_2$;

$\mrm{(e)}$
$w$ and $n$ are odd, and $\ell> C_2$.
\end{theorem}
\begin{theorem}
\label{Cor2}
Suppose that $w$ is odd or $w> 2r$.
If $\ell$ does not split in $K$,
then the set 
$\mrm{Rep}_{\mbb{Q}_{\ell}}
(G_K)^{\bullet}$ 
is empty under
any one of the following situations:

$\mrm{(a)}$ 
$w$ is odd, $\ell\nmid d_K$ and $\ell > C'_1$;

$\mrm{(b)}$ 
$w$ is odd, the extension $K/\mbb{Q}$ has odd degree
and  $\ell > C'_2$;

$\mrm{(c)}$ 
$w>2r$, $\ell\nmid d_K$ and $\ell > C'_1$;

$\mrm{(d)}$
$w>2r$ and $\ell> C'_2$;

$\mrm{(e)}$
$w$ and $n$ are odd, and $\ell> C'_2$.

\end{theorem}

\begin{proof}[Proofs of Theorem \ref{Cor1} and \ref{Cor2}]
We only prove Theorem \ref{Cor1} because
we can prove Theorem \ref{Cor2} by the same way.
Suppose that 
there exists an $\ell$-adic Galois representation $V$ 
which is contained in 
$\mrm{Rep}_{\mbb{Q}_{\ell}}
(G_K)^{\bullet}_{\mrm{cycl}}$
and take its residual representation $\bar V$.
If we assume one of the situations (a) and (b) given in Proposition \ref{Thm1},
then $\bar V$ is of uniform tame inertia weight
$e_{\lambda}w/2$ at any place ${\lambda}$ of $K$ above $\ell$, 
and thus $e_{\lambda}w/2$ must be a rational integer.
Moreover, by Theorem \ref{Ca},
it follows that the tame inertia weight $e_{\lambda}w/2$ is
between $0$ and $e_{\lambda}r$.    
However, if we assume any one of the conditions
(a), (b), (c) and (d),
then $e_{\lambda}w$ is odd for some ${\lambda}$ or $e_{\lambda}w/2>e_{\lambda}r$.
This is a contradiction. 
The rest of the assertion related with (e) follows 
from the fact (\cite{CS}, Theorem 1) that
the sum of all the tame inertia weights of $V$ at ${\lambda}$
must be divisible by $e_{\lambda}$.
\end{proof}

\begin{remark}
To remove the special assumption ``$\ell$ does not split in $K$''
in Theorem \ref{Cor2}
is impossible in general because there exists such an example,
which is pointed out by Akio Tamagawa:
Let $E$ be an elliptic curves over $K$ 
with complex multiplication over $K$ by an imaginary quadratic field 
$F:=\mbb{Q}\otimes_{\mbb{Z}} \mrm{End}_K(E)\subset K$.
Then $E$ is potential everywhere good reduction and 
thus we may suppose 
$E$ has everywhere good reduction over $K$.  
Put $F_{\ell}:=\mbb{Q}_{\ell}\otimes_{\mbb{Q}} F$,
which is a semisimple $\mbb{Q}_{\ell}$-algebra.
It is well-known that 
$F_{\ell}$ acts faithfully on the Tate-module 
$V_{\ell}(E)$ of $E$ and thus 
$V_{\ell}(E)$ has a natural structure of 
$1$-dimensional $F_{\ell}$-vector space. 
If $\ell$ splits in $F$, the decomposition 
$F_{\ell}\simeq \mbb{Q}_{\ell}\times \mbb{Q}_{\ell}$
induces a decomposition of $V_{\ell}(E)$ as a sum of $1$-dimensional 
$G_K$-stable $\ell$-adic representations.
For such odd prime $\ell$, 
it is easy to check that $V_{\ell}(E)$ is an element of the set
$\mrm{Rep}_{\mbb{Q}_{\ell}}
(G_K)^{\bullet}$, 
where $\bullet=(2, 2,1,1)$. 
\end{remark}
 
\section{Applications}

We give some applications of our results.
We use same notation as in the previous section.

\subsection{Rasmussen-Tamagawa Conjecture}
As a first application, 
we show a special case of a Conjecture 
of Rasmussen and Tamagawa.
We denote by $\tilde{K_{\ell}}$ the maximal pro-$\ell$ extension 
of $K(\mu_{\ell})$ which is unramified away from $\ell$. 
\begin{definition}
Let $g\ge 0$ be an integer.
We denote by $\mcal{A}(K,g,\ell)$
the set of $K$-isomorphism classes of abelian varieties $A$ over $K$,
of dimension $g$, which satisfy the following equivalent 
conditions:

\noindent
(1) $K(A[\ell^{\infty}])\subset \tilde{K_{\ell}}$;

\noindent
(2) The abelian variety $A$ has good reduction outside $\ell$ and 
the extension $K(A[\ell])/K(\mu_{\ell})$ is an $\ell$-extension;

\noindent
(3) The abelian variety $A$ has good reduction outside $\ell$ and
$A[\ell]$ admits a filtration of $G_K$-modules
\[
\{0\}= \bar{V}_0\subset \bar{V}_1 \subset \dots \subset
\bar{V}_{2g-1}\subset \bar{V}_{2g}=A[\ell]
\]
such that $\bar{V}_k$ has dimension $k$ for each $1\le k\le 2g$. 
Furthermore, for each $1\le k\le 2g$, the $G_K$-action on the space 
$\bar{V}_k/\bar{V}_{k-1}$ is given by 
$g.\bar v=\chi_{\ell} (g)^{a_k}\cdot \bar v$ 
for some $a_k\in \mbb{Z}$.
\end{definition}

\noindent
The equivalently of the above three conditions 
follows from the criterion of N\'eron-Ogg-Shafarevich 
and Lemma \ref{RTlem} below
(put $G=\mrm{Gal}(K(A[\ell^{\infty}])/K)$,
$N=\mrm{Gal}(K(A[\ell^{\infty}])/K(\mu_{\ell}))$
and apply Lemma \ref{RTlem} to 
the group $A[\ell]$).  
The set $\mcal{A}(K,g,\ell)$ is a finite set because of Faltings' proof 
of Shafarevich Conjecture.  
Rasmussen and Tamagawa conjectured 
that for any $\ell$ large enough,
this set is empty: 
\begin{conjecture}[\cite{RT}, Conjecture 1]
\label{RTc}
The set $\mcal{A}(K,g):=\{(A,\ell)\mid 
[A]\in \mcal{A}(K,g,\ell),\ \ell:\mrm{prime\ number }\}$
is finite, that is, the set $\mcal{A}(K,g,\ell)$ is empty  
for any prime $\ell$ large enough. 
\end{conjecture}

\noindent
We call this conjecture the 
\textit{Rasmussen-Tamagawa conjecture}. 
It is known that the Rasmussen-Tamagawa conjecture 
holds under the following conditions:

\vspace{2mm}

\noindent 
(i) $K=\mbb{Q}$ and $g=1$ (\cite{RT}, Theorem 2);

\vspace{2mm} 

\noindent 
(ii) $K$ is a quadratic number field other than the 
imaginary quadratic fields of class number one 
and $g=1$ (\cite{RT}, Theorem 4). 

\vspace{2mm}

\noindent 
We consider the semistable reduction case of 
Conjecture \ref{RTc}.
\begin{definition}
(1) We denote by $\mcal{A}(K,g,\ell)_{\mrm{st}}$
the set of $K$-isomorphism classes of 
abelian varieties in  
$\mcal{A}(K,g,\ell)$
with everywhere semistable reduction.

\noindent
(2) We denote by  $\mcal{A}(K,g,\ell_0,\ell)_{\mrm{st}}$
the set of $K$-isomorphism classes of abelian varieties $A$ over $K$
with everywhere semistable reduction,
of dimension $g$, which satisfy the following condition:
The abelian variety $A$ has good reduction 
at some places of $K$ above $\ell_0$ 
and
$A[\ell]$ admits a filtration of $G_K$-modules
\[
\{0\}= \bar{V}_0\subset \bar{V}_1 \subset \dots \subset
\bar{V}_{2g-1}\subset \bar{V}_{2g}=A[\ell]
\]
such that $\bar{V}_k$ has dimension $k$ for each $1\le k\le 2g$. 
\end{definition}

\noindent
Clearly, 
we see $\mcal{A}(K,g,\ell)_{\mrm{st}}\subset 
\mcal{A}(K,g,\ell_0,\ell)_{\mrm{st}}$ since $\ell\not= \ell_0$.
The set $\mcal{A}(K,g,\ell)_{\mrm{st}}$ is finite, however,
the set $\mcal{A}(K,g,\ell_0,\ell)_{\mrm{st}}$ may be infinite.
The Rasmussen-Tamagawa conjecture implies that 
$\mcal{A}(K,g,\ell)_{\mrm{st}}$ will be empty 
for a prime $\ell$ large enough.
We will prove that $\mcal{A}(K,g,\ell_0,\ell)_{\mrm{st}}$
is in fact empty for a prime $\ell$ large enough
which does not split in $K$.
Recall the lemma proved by Rasmussen and Tamagawa
(cf.\ \cite{RT}, Lemma 3). 
Let $G$ be a topological group 
with a normal pro-$\ell$ open subgroup $N$, such that 
the quotient $\Delta=G/N$ is isomorphic to a subgroup of 
$\mbb{F}_{\ell}^{\times}$.
Because $N$ is pro-$\ell$,
we see that $N$ has trivial image under any character 
$\psi\colon G\to \mbb{F}_{\ell}^{\times}$.
Hence, there always exists an induced character   
$\bar \psi\colon \Delta\to \mbb{F}_{\ell}^{\times}$. 
Let $\chi\colon G \to \mbb{F}_{\ell}^{\times}$  
be a character such that the induced 
character $\bar \chi$ is an injection
$\Delta\hookrightarrow \mbb{F}_{\ell}^{\times}$.
Finally, let $\bar V$ be a finite dimensional 
$\mbb{F}_{\ell}$-vector space of dimension $n$ 
on which $G$ acts continuously.

\begin{lemma}
\label{RTlem}
The vector space $V$ admits a filtration of $G_K$-modules
\[
\{0\}= \bar{V}_0\subset \bar{V}_1 \subset \dots \subset
\bar{V}_{n-1}\subset \bar{V}_n=\bar V
\]
such that $\bar{V}_k$ has dimension $k$ for each $1\le k\le n$. 
Furthermore, for each $1\le k\le n$, the $G$-action on the space 
$\bar{V}_k/\bar{V}_{k-1}$ is given by 
$g.\bar v=\chi (g)^{a_k}\cdot \bar v$ 
for some $a_k\in \mbb{Z}$, 
$0\le a_k< \# \Delta$.
\end{lemma}

\begin{proof}
The proof will proceed by the same method as the 
proof of Lemma 3 of \cite{RT}, thus we omit it.
\end{proof}

\noindent
Take an abelian variety $A$ which is in the set
$\mcal{A}(K,g,\ell)_{\mrm{st}}$
(resp.\ $\mcal{A}(K,g,\ell_0,\ell)_{\mrm{st}}$). 
Then $V_{\ell}(A)$  
is an element of the set 
$\mrm{Rep}_{\mbb{Q}_{\ell}}
(G_K)^{\bullet}_{\mrm{cycl}}$ 
(resp.\ $\mrm{Rep}_{\mbb{Q}_{\ell}}
(G_K)^{\bullet}$ )
with $\bullet=(2g,2,1,1)$ 
(resp.\ $\bullet=(2g,\ell_0,1,1)$)
for any $\ell>2$ (resp.\ $\ell>\ell_0$).
Consequently, we obtain the following results as 
corollaries of Theorem \ref{Cor1} and \ref{Cor2}:

\begin{corollary}
\label{RTst}
The set $\mcal{A}(K,g,\ell)_{\mrm{st}}$ is empty 
under any one of the following situations:

$\mrm{(a)}$ $\ell\nmid d_K$ and 
$\ell > 
2^{\delta_1}
\left(\begin{smallmatrix}
  2g \\ g
\end{smallmatrix}\right)$,
where $\delta_1:=2dg+1$;

$\mrm{(b)}$ 
The extension $K/\mbb{Q}$ has odd degree
and
$\ell > 
2^{\delta_2}
\left(\begin{smallmatrix}
  2g \\ g
\end{smallmatrix}\right)$,
where $\delta_2:=2d^2g+1$.
\end{corollary}

\begin{corollary}
\label{GRTst}
Suppose that $\ell$ does not split in $K$.
The set $\mcal{A}(K,g,\ell_0,\ell)_{\mrm{st}}$ is empty 
under any one of the following situations:

$\mrm{(a)}$ $\ell\nmid d_K$ and 
$\ell > 
2\ell_0^{\delta_1'}
\left(\begin{smallmatrix}
  2g \\ g
\end{smallmatrix}\right)$,
where $\delta_1':=2dgh_K^+$;

$\mrm{(b)}$ 
The extension $K/\mbb{Q}$ has odd degree
and
$\ell > 
2\ell_0^{\delta_2'}
\left(\begin{smallmatrix}
  2g \\ g
\end{smallmatrix}\right)$,
where $\delta_2':=2d^2gh_K^+$.

\end{corollary}
\vspace{2mm}

\begin{remark}
Rasmussen and Tamagawa have shown the 
finiteness of the set  $\mcal{A}(K,g)_{\mrm{st}}$
by using the result of \cite{Ra} instead of Theorem \ref{Ca}
(unpublished).
Our main results in this paper are motivated by their work. 
\end{remark}

\subsection{Irreducibility of $\ell$-torsion points of elliptic curves}

We consider the following classical question:

\begin{question} 
Does there exist a constant $c_K$, which depends only on $K$, 
such that for any semistable elliptic curve $E$ defined over $K$
without complex multiplication over $K$, 
the representation in its $\ell$-torsion points $E[\ell]$ 
is irreducible whenever $\ell>c_K?$ 
Furthermore, if the answer is positive, 
how can we evaluate such a constant $c_K?$ 
\end{question} 

\noindent 
By Mazur's results on a moduli of rational points of modular curve 
$X_0(N)$ (\cite{Ma}), it is known that 
$c_\mbb{Q}=7$.  
If $K$ is a quadratic field, 
then the existence of $c_K$ is known 
and moreover, 
if the class number of $K$ is $1$, then the 
explicit calculation of $c_K$ is given by 
Kraus \cite{Kr1}. 
By combining results on Merel (\cite{Me}) and 
Momose (\cite{Mo}), 
Kraus showed 
the existence of $c_K$  
for a number field $K$ which does not contain 
an imaginary quadratic field of class number $1$ (\cite{Kr2}).  
Moreover, Kraus defined the good condition ``(C)'' 
associated with $K$ in {\it op.\ cit,} such that
the existence and the explicit value of $c_K$ is known
if $K$ satisfies this condition.   

The following is easy consequence of 
Corollary \ref{GRTst} under the case $g=1$.

\begin{corollary}
\label{ell}
Let $E$ be an elliptic curve over $K$ with 
everywhere semistable reduction.
Let $\ell_E$ be the minimal prime number $p$ such that 
$E$ has good reduction at some finite places 
of $K$ above $p$. 
Suppose $\ell$ does not split in $K$.
Then 
$E[\ell]$
is irreducible 
under any one of the following conditions:

$\mrm{(a)}$ $\ell\nmid d_K$ and 
$\ell> 4\ell_E^{\delta_1''}$, where
$\delta_1'':=2dh_K^+$;

$\mrm{(b)}$ 
The extension $K/\mbb{Q}$ has odd degree
and
$\ell > 
4\ell_E^{\delta_2''}$,
where $\delta_2'':=2d^2h_K^+$.
\end{corollary}

We remark that
the above corollary is valid 
even if $E$ has complex multiplication over $K$.

\subsection{Residual properties of \'etale cohomology groups}

For any semistable elliptic curve $E$
over $\mbb{Q}$, Serre proved the following
(\cite{Se}, Section 5.4, Proposition 21, Corollary 1):
Let $\ell_E$ be the minimal prime number $p$ such that 
$E$ has good reduction at $p$.    
Then $E[\ell]$ is irreducible 
if $\ell>(1+\ell_E^{1/2})^2$.

As a corollary of Theorem \ref{Cor2}, we can slightly generalize this fact 
to \'etale cohomology groups of odd degree. 

\begin{corollary}
\label{Et}
Let $X$ be a proper smooth scheme over $K$
with everywhere semistable reduction 
and $w$ an odd integer.
Let $b_w(X)$ be a $w$-th Betti number of $X$ and 
$\ell_X$ the minimal prime number $p$ such that 
$X$ has good reduction at some places of $K$ above $p$.    
Then there exists a constant $C$ depending only on 
$b_w(X)$ and $\ell_X$ such that 
for any prime number $\ell>C$ which does not split in $K$,
the \'etale cohomology group
$H^w_{\mrm{\acute{e}t}}(X_{\bar K}, \mbb{Q}_{\ell})$
is not of residually Borel.
More precisely, if $\ell$ does not split in $K$,
$H^w_{\mrm{\acute{e}t}}(X_{\bar K}, \mbb{Q}_{\ell})$
is not of residually Borel
under any one of the following conditions:

$\mrm{(a)}$ $\ell\nmid d_K$ and 
$\ell> 2B_{b_w(X)}\ell_X^{\Delta_1}$, where
$\Delta_1:=b_w(X)dh_K^+w$;

$\mrm{(b)}$ 
The extension $K/\mbb{Q}$ has odd degree
and
$\ell > 
2B_{b_w(X)}\ell_X^{\Delta_2}$,
where $\Delta_2:=b_w(X)d^2h_K^+w$.

\end{corollary}

\begin{proof}
Putting
$\bullet:=(b_w(X),\ell_X,w,w)$,
we see that the dual of 
$H^w_{\mrm{\acute{e}t}}(X_{\bar K}, \mbb{Q}_{\ell})$
is contained in the set
$\mrm{Rep}_{\mbb{Q}_{\ell}}
(G_K)^{\bullet}$.
Applying Theorem \ref{Cor2},   
we obtain the desired result.
\end{proof}

For any proper smooth scheme $X$ over $K$,
there exists an finite extension $L$ over $K$ such that 
$H^w_{\mrm{\acute{e}t}}(X_{\bar L}, \mbb{Q}_{\ell})$ 
is everywhere semistable as a representation of $G_L$
for almost all $\ell$
by  Proposition \ref{var}.
If this is the case, we see that 
$H^w_{\mrm{\acute{e}t}}(X_{\bar L}, \mbb{Q}_{\ell})^{\vee}$
satisfies $\mrm{(G}$-$\mrm{1)}$,  $\mrm{(G}$-$\mrm{2)}$ and 
$\mrm{(G}$-$\mrm{3)}$ as a representation of $G_L$.
Thus if we can obtain the explicit description 
of $L$, 
we will able to obtain the analogous result of
corollary  \ref{Et} for a prime $\ell$ large enough 
which does not split in $L$.
However, it is very difficult to determine 
such $L$ in general.
We can determine this $L$ if $X$ is an abelian variety.
If this is the case, 
Raynaud's criterion of semistable reduction 
(\cite{Gr}, Proposition 4.7)
implies that $X$ is everywhere semistable reduction over
$L:=K(X[3],X[5])$.

\end{document}